\documentclass[11pt,a4paper,leqno]{article}
\usepackage[margin=1.1in]{geometry}
\usepackage{amsmath,amsthm,amssymb,enumerate, bbm ,graphicx,color,caption,upgreek, float, wasysym, subcaption,booktabs,longtable, appendix,subcaption,graphics, pdfpages,rotating,mathtools,subcaption,array,supertabular}
\usepackage[justification=justified,singlelinecheck=false]{caption}
\usepackage{MnSymbol}

\usepackage{tkz-graph}
\tikzset{
  VertexStyle/.append style = {shape=circle,draw, fill=black, minimum size=9pt, inner sep=-1pt},
  EdgeStyle/.append style = {-, thick},
  LoopStyle/.append style = {-}}
\usetikzlibrary{arrows.meta}
\tikzset{>={Latex[width=2.5mm,length=2.5mm]}}

\usepackage{algorithm}
\usepackage{algorithmic}
\floatname{algorithm}{Algorithm}

\definecolor{donkergroen}{RGB}{46,148,0}
\usepackage[hidelinks]{hyperref} 
\definecolor{donkerrood}{RGB}{204,0,0}

\definecolor{blauw}{RGB}{61,158,255}
\definecolor{donkerblauw}{RGB}{0,0,255}
\definecolor{donkergroen}{RGB}{46,148,0}
\definecolor{donkerrood}{RGB}{204,0,0}

\usepackage{abstract}

\makeatletter
\newif\if@borderstar
\def\bordermatrix{\@ifnextchar*{%
\@borderstartrue\@bordermatrix@i}{\@borderstarfalse\@bordermatrix@i*}%
}
\def\@bordermatrix@i*{\@ifnextchar[{\@bordermatrix@ii}{\@bordermatrix@ii[()]}}
\def\@bordermatrix@ii[#1]#2{%
\begingroup
\m@th\@tempdima8.75\p@\setbox\z@\vbox{%
\def\cr{\crcr\noalign{\kern 2\p@\global\let\cr\endline }}%
\ialign {$##$\hfil\kern 2\p@\kern\@tempdima & \thinspace %
\hfil $##$\hfil && \quad\hfil $##$\hfil\crcr\omit\strut %
\hfil\crcr\noalign{\kern -\baselineskip}#2\crcr\omit %
\strut\cr}}%
\setbox\tw@\vbox{\unvcopy\z@\global\setbox\@ne\lastbox}%
\setbox\tw@\hbox{\unhbox\@ne\unskip\global\setbox\@ne\lastbox}%
\setbox\tw@\hbox{%
$\kern\wd\@ne\kern -\@tempdima\left\@firstoftwo#1%
\if@borderstar\kern2pt\else\kern -\wd\@ne\fi%
\global\setbox\@ne\vbox{\box\@ne\if@borderstar\else\kern 2\p@\fi}%
\vcenter{\if@borderstar\else\kern -\ht\@ne\fi%
\unvbox\z@\kern-\if@borderstar2\fi\baselineskip}%
\if@borderstar\kern-2\@tempdima\kern2\p@\else\,\fi\right\@secondoftwo#1 $%
}\null \;\vbox{\kern\ht\@ne\box\tw@}%
\endgroup
}
\makeatother

\makeatletter 
\newcommand\mynobreakpar{\par\nobreak\@afterheading} 
\makeatother

\newcommand{\N}{\mathbb{N}}
\newcommand{\Z}{\mathbb{Z}}
\newcommand{\D}{\mathbb{D}}

\newcommand{\R}{\mathbb{R}}

\usepackage[english]{babel}
\newtheorem{conjecture}{Conjecture}[section]
\newtheorem{theorem}{Theorem}[section]
\newtheorem{lemma}[theorem]{Lemma}
\newtheorem{claim}[theorem]{Claim}
\newtheorem{proposition}[theorem]{Proposition}
\newtheorem{corollary}[theorem]{Corollary}
\newtheorem{question}[theorem]{Question}

\theoremstyle{definition}

\newtheorem{examp}{Example}[section]
\newtheorem*{examp*}{Example}

\theoremstyle{plain}

\newcounter{thm}[section]

\title{{\large \textbf{ON DIHEDRAL FLOWS IN EMBEDDED GRAPHS}}}
\author{Bart Litjens\thanks{University of Amsterdam, Netherlands. Email: \texttt{bart\_litjens@hotmail.com}. Supported by the European Research Council under the European Union's Seventh Framework Programme (FP7/2007-2013) / ERC grant agreement n$\mbox{}^{\circ}$ 339109.}}
\date{\vspace{-5ex}}
\begin{document}
\maketitle

\noindent {\small \textbf{Abstract.} Let $\Gamma$ be a multigraph with for each vertex a cyclic order of the edges incident with it. For $n \geq 3$, let $D_{2n}$ be the dihedral group of order $2n$. Define $\D := \{\begin{psmallmatrix} \pm 1 & a \\ 0 & 1 \end{psmallmatrix} \mid a \in \Z\}$. Goodall, Krajewski, Regts and Vena in 2016 asked whether $\Gamma$ admits a nowhere-identity $D_{2n}$-flow if and only if it admits a nowhere-identity $\D$-flow with $|a| < n$ (a `nowhere-identity dihedral $n$-flow'). We give counterexamples to this statement and provide general obstructions. Furthermore, the complexity of deciding the existence of nowhere-identity $2$-flows is discussed. Lastly, graphs in which the equivalence of the existence of flows as above is true, are described. We focus particularly on cubic graphs.}\\

\noindent \textbf{Key words:} embedded graph, cubic graph, dihedral group, flow, nonabelian flow\\
\noindent \textbf{MSC 2010:} 05C10, 05C15, 05C21, 20B35

\section{Introduction}

Let $\Gamma = (V,E)$ be a multigraph and $G$ an additive abelian group. A \textit{nowhere-zero} $G$\textit{-flow} on $\Gamma$ is an assignment of non-zero group elements to edges of $\Gamma$ such that, for some orientation of the edges, Kirchhoff's law is satisfied at every vertex. The following existence theorem is due to Tutte \cite{tutte49}.

\begin{theorem}\label{theorem:abelian}
Let $\Gamma$ be a multigraph and $n \in \N$. The following are equivalent statements.
\begin{enumerate}
\item There exists a nowhere-zero $G$-flow on $\Gamma$, for each abelian group $G$ of order $n$.
\item There exists a nowhere-zero $G$-flow on $\Gamma$, for some abelian group $G$ of order $n$.
\item There exists a nowhere-zero $\Z$-flow on $\Gamma$ with values in $\{\pm 1, ..., \pm(n-1)\}$.
\end{enumerate}
\end{theorem}

\noindent A flow as in the third statement above is called a \textit{nowhere-zero} $n$\textit{-flow}. Generalizing the equivalence of the first two statements of Theorem \ref{theorem:abelian}, Tutte showed in \cite{tutte54}:

\begin{theorem}\label{theorem:number}
For each multigraph $\Gamma$, there exists a polynomial $P_{\Gamma}$ such that for every finite abelian group $G$ of order $n$, the number of nowhere-zero $G$-flows on $\Gamma$ is $P_{\Gamma}(n)$.
\end{theorem}

\noindent Kochol showed that the number of nowhere-zero $n$-flows is generally larger than the number of nowhere-zero $G$-flows, but still is a polynomial in $n$ \cite{kochol02}.\\
\indent For graphs, flows with values in a finite nonabelian group are generally not well-defined. For Kirchhoff's law to be unambiguously defined requires specifying for each vertex a cyclic order on those edges incident with it. The collection of cyclic orders at vertices forms a rotation system for the graph. Thus, flows with values in a nonabelian group are well-defined for a cellularly embedded graph in a closed orientable surface. By augmenting such a vertex rotation system by a signature on edges, it is furthermore possible to define flows for graphs embedded in non-orientable surfaces. In that case, edges receive bi-orientations rather than orientations.\\
\indent The analogue of Theorem \ref{theorem:number} in this more general setting is false. If a graph $\Gamma$ is embedded in a closed orientable surface $S$, then the number of nowhere-identity $G$-flows (with $G$ a finite multiplicative group) depends on the Euler characteristic of $S$ and the multiset of dimensions of irreducible representations of $G$ \cite{goodall16,litjens17}. In case $\Gamma$ is embedded in a closed non-orientable surface $S$, then the trichotomy of types of irreducible representations according to the Frobenius-Schur indicator moreover determines the number of nowhere-identity $G$-flows \cite{goodall17}.\\
\indent It is not immediately clear how a counterpart of Theorem \ref{theorem:abelian} would look like for nonabelian groups. In \cite{goodall16} a question is posed for the \textit{dihedral group} $D_{2n}$, with $n \geq 3$. By $\Z_n$ we denote the cyclic group of order $n$. Identify $D_{2n} = \langle r,s \mid r^n = srsr = s^2 = 1 \rangle$ with the group
\begin{equation}\label{equation:dihedraal}
\left\{\begin{pmatrix} \pm 1 & a \\ 0 & 1 \end{pmatrix} \mid a \in \Z_n\right\}, \text{ via } r = \begin{pmatrix} 1 & 1 \\ 0 & 1 \end{pmatrix} \text{ and } s = \begin{pmatrix} -1 & 0 \\ 0 & 1 \end{pmatrix}.
\end{equation}
Elements of the subgroup $\langle r \rangle$ of $D_{2n}$ generated by $r$ are called \textit{rotations}. They form a group isomorphic to $\Z_n$. Elements of the cosets $s\langle r \rangle$ in $D_{2n}$ are called \textit{reflections}. Each reflection is of order $2$. We define 
\[
\D := \left\{\begin{pmatrix} \pm 1 & a \\ 0 & 1 \end{pmatrix} \mid a \in \Z\right\} \hspace{1mm} \text{and} \hspace{1mm} \D^{<n} := \left\{\begin{pmatrix} \pm 1 & a \\ 0 & 1 \end{pmatrix} \mid a \in \Z, |a| < n\right\}.
\]
\noindent Let 
\begin{equation}\label{equation:pin}
\Pi_n: \D \rightarrow D_{2n},
\end{equation}
denote the group homomorphism that reduces the upper right matrix entry modulo $n$. Elements in $\D$ that are mapped to rotations in $D_{2n}$ by $\Pi_n$ are called rotations, and likewise for reflections.\\
\indent Let $\Gamma$ be a multigraph embedded in a closed orientable surface. A nowhere-identity $\D$-flow on $\Gamma$ (defined formally in Section \ref{subsection:flow}) which takes values in $\D^{< n}$, is called a \textit{nowhere-identity dihedral $n$-flow}. The authors of \cite{goodall16} asked whether the following statement is true:
\begin{align}\label{equation:main}
\Gamma \textit{ admits a nowhere-identity }D_{2n}\textit{-flow } \Longleftrightarrow \hspace{1mm}\Gamma \textit{ admits a nowhere-identity dihedral $n$-flow}.
\end{align}
The ``$\Longleftarrow$" direction is easily seen to hold. Indeed, observe that by composition with $\Pi_n$, a nowhere-identity dihedral $n$-flow yields a nowhere-identity $D_{2n}$-flow. In this paper we show that the ``$\Longrightarrow$" direction is false in general. A counterexample to statement (\ref{equation:main}) for the case $n=3$ is given by the following graph (see Section \ref{subsection:counterexamples} for more details):
\vspace{-6mm}
\begin{figure}[H]
\centering
\begin{tikzpicture}
\SetGraphUnit{1.4}
\Vertex{a}
\EA(a){b}
\NOWE(a){c}
\SOWE(a){d}
\draw[style={-,thick,color=black}] (a)--(b)node[pos=0.5,anchor=south]{};
\draw[style={-,thick,color=black}] (a)--(c)node[pos=0.5,anchor=south]{};
\draw[style={-,thick,color=black}] (a)--(d)node[pos=0.5,anchor=north]{};
\draw[style={-,thick,color=black}] (c) edge[in=180,out=260,loop,min distance=20mm] node[below] {} (c);
\draw[style={-,thick,color=black}] (c) edge[in=140,out=220,loop,min distance=20mm] node[below] {} (c);
\draw[style={-,thick,color=black}] (d) edge[in=100,out=180,loop,min distance=20mm] node[below] {} (d);
\draw[style={-,thick,color=black}] (d) edge[in=140,out=220,loop,min distance=20mm] node[below] {} (d);
\draw[style={-,thick,color=black}] (b) edge[in=300,out=20,loop,min distance=20mm] node[below] {} (b);
\draw[style={-,thick,color=black}] (b) edge[in=340,out=60,loop,min distance=20mm] node[below] {} (b);
\end{tikzpicture}
\vspace{-5mm}
\caption{An embedded graph with a nowhere-identity $D_6$-flow but with no nowhere-identity dihedral $3$-flow} 
\end{figure}
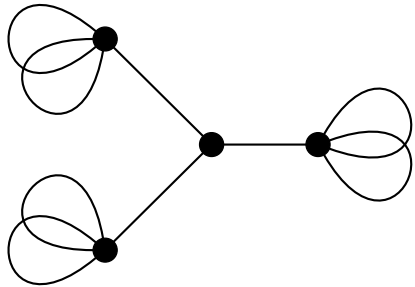
\indent The paper is organized as follows. In Section \ref{section:def} we give definitions and preliminaries that are used throughout the article. Section \ref{section:neg} focusses on counterexamples and obstructions to statement (\ref{equation:main}). Also, the complexity of deciding the existence of nowhere-identity dihedral $2$-flows is addressed. Section \ref{section:pos} considers cubic graphs for which statement (\ref{equation:main}) holds. Special attention is devoted to snarks. Examples are used throughout to illustrate results. Lastly, we discuss material for further research in Section \ref{section:fr}.

\section{Definitions and preliminaries}\label{section:def}

In this section we give definitions and preliminaries. These mostly involve notions of standard flow theory, tailored to the setting of embedded graphs, for which the necessary background may be found in \cite{mohar01}. Furthermore, we set up notation that is used throughout the article. In this paper, all graphs are assumed to be connected, but they are not assumed to be simple.

\subsection{Graph embeddings and flows}\label{subsection:flow}
Let $\Gamma = (V,E)$ be a graph with for each vertex $v$ a cyclic order $\pi_v$ of the edges incident with~$v$. A loop appears twice in the cyclic order. We define the \textit{rotation system} $\pi$ of $\Gamma$ by $\pi := \{\pi_v \mid v \in V\}$. A rotation system of $\Gamma$ is equivalent to a cellular embedding of $\Gamma$ in a closed connected orientable surface $S$ \cite[Thm.\ $3.2.2$]{mohar01}. If $S$ has genus $g$, we will also refer to it as the \textit{torus with }$g$ \textit{holes}, or the $g$\textit{-torus}. The collection of faces of $\Gamma$ (components of $S \setminus \Gamma$ homeomorphic to an open disc in $\R^2$) is denoted by $F$. By Euler's formula 
\begin{equation}\label{equation:euler}
|V|-|E|+|F| = 2-2g.
\end{equation}
From now on, whenever we speak of an embedded graph, we mean a graph that is cellularly embedded in a closed orientable surface.\\
\indent Let $G$ be a group and let $\Gamma = (V,E)$ be a graph with rotation system $\{\pi_v \mid v \in V\}$. Assume that an arbitrary orientation is given to the edges of $\Gamma$. A $G$\textit{-flow} on $\Gamma$ is an assignment of values of $G$ to the edges of $\Gamma$, such that for each vertex $v$ the product of the values on its incident edges (in the order $\pi_v$) equals the identity, where the inverse is taken for any edge leaving $v$. A $G$-flow that on no edge evaluates to the identity, is called a \textit{nowhere-identity} $G$\textit{-flow}. If $\Gamma$ admits a nowhere-identity $G$-flow with respect to some orientation, then it admits such a flow with respect to any orientation, as can be seen by inverting the flow values on the appropriate set of edges.

\section{Negative results}\label{section:neg}

In this section we show that statement (\ref{equation:main}) is false and discuss to what extent it fails. 

\subsection{Counterexamples}\label{subsection:counterexamples}

An important difference between flows on embedded graphs taking values in an abelian group or in a nonabelian group is that for the latter, bridges are not an obstruction to the existence of a flow. Let $G$ be a group with commutator subgroup $[G,G]$, and let $\pi_{\text{ab}}: G \rightarrow G/[G,G]$ denote the canonical homomorphism. The commutator subgroup of $D_{2n}$ is $\langle r^2 \rangle$, with $r$ as in equation (\ref{equation:dihedraal}), as $sr^k(sr^k)^{-1} = r^{2k}$. Thus
\[
|[D_{2n},D_{2n}]| = \left|\left\langle r^2 \right\rangle\right| = \left\{\begin{array}{ll} n/2 & \text{if $n$ is even},\\ n & \text{if $n$ is odd.}\end{array}\right.
\]
Let $\Gamma = (V,E)$ be a graph. For $X \subseteq V$, let $\delta(X)$ be the set of edges with one end in $X$ and the other in $V\setminus X$.

\begin{lemma}\label{lemma:commsubgroup}
Let $\Gamma = (V,E)$ be an embedded graph and $f: E \rightarrow G$ a $G$-flow of $\Gamma$. Then, for every $X \subseteq V$ we have that 
\begin{equation}\label{equation:incomm}
\prod_{e \in \delta(X)}f(e) \in [G,G],
\end{equation}
for any order of the edges $e \in \delta(X)$ in the product above.
\end{lemma}
\begin{proof}
Consider the flow with values in the abelian group $G/[G,G]$ given by the composition $\pi_{\text{ab}} \circ f$. As abelian flows are zero on cut-sets \cite[Prop.\ $6.1.1$]{diestel00}, the expression on the left-hand side of (\ref{equation:incomm}) lies in the kernel of $\pi_{\text{ab}}$, which equals $[G,G]$.
\end{proof}

\noindent Lemma \ref{lemma:commsubgroup} shows that embedded graphs with bridges may have nowhere-identity $G$-flows, provided that $G$ is nonabelian. Indeed, consider the following embedded graph $\Gamma = (V,E)$ whose vertices are given a clockwise cyclic order:
\vspace{-6mm}
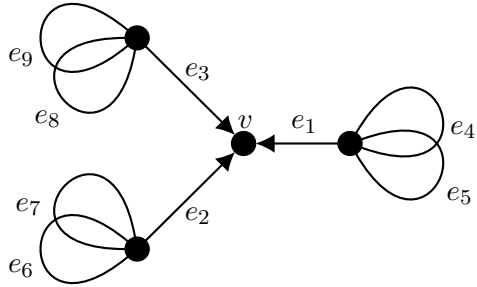
\begin{figure}[H]
\centering
\begin{tikzpicture}
\SetGraphUnit{1.4}
\Vertex{a}
\EA(a){b}
\NOWE(a){c}
\SOWE(a){d}
\draw[style={<-,thick,color=black}] (a)--(b)node[pos=0.5,anchor=south]{\hspace{-5mm}$v$\hspace{5mm}$e_1$};
\draw[style={<-,thick,color=black}] (a)--(c)node[pos=0.5,anchor=south]{\hspace{2mm}$e_3$};
\draw[style={<-,thick,color=black}] (a)--(d)node[pos=0.5,anchor=north]{\hspace{2mm}$e_2$};
\draw[style={-,thick,color=black}] (c) edge[in=180,out=260,loop,min distance=20mm] node[below] {\vspace{5mm}\hspace{-4mm}$e_8$} (c);
\draw[style={-,thick,color=black}] (c) edge[in=140,out=220,loop,min distance=20mm] node[below] {\vspace{-5mm}\hspace{-5mm}$e_9$} (c);
\draw[style={-,thick,color=black}] (d) edge[in=100,out=180,loop,min distance=20mm] node[below] {\vspace{5mm}\hspace{-9mm}$e_7$} (d);
\draw[style={-,thick,color=black}] (d) edge[in=140,out=220,loop,min distance=20mm] node[below] {\vspace{-5mm}\hspace{-5mm}$e_6$} (d);
\draw[style={-,thick,color=black}] (b) edge[in=300,out=20,loop,min distance=20mm] node[below] {\vspace{5mm}\hspace{5mm}$e_{5}$} (b);
\draw[style={-,thick,color=black}] (b) edge[in=340,out=60,loop,min distance=20mm] node[below] {\vspace{-5mm}\hspace{6mm}$e_{4}$} (b);
\end{tikzpicture}
\vspace{-5mm}
\caption{$\Gamma$: a counterexample to statement (\ref{equation:main})} \label{fig:M1}
\end{figure}
\vspace{-1mm}
\noindent Note that the graph $\Gamma$ with this vertex rotation system has precisely one face and is therefore an embedding of $\Gamma$ in the $3$-torus. By Lemma \ref{lemma:commsubgroup}, any nowhere-identity dihedral $3$-flow has to assign a non-identity value from 
\begin{equation}\label{equation:intersection}
[\D,\D] \cap \D^{<3} = \left\{\begin{pmatrix} 1 & 0 \\ 0 & 1 \end{pmatrix},\begin{pmatrix} 1 & \pm 2 \\ 0 & 1 \end{pmatrix}\right\},
\end{equation}
to the edges $e_1,e_2$ and $e_3$. But then clearly Kirchhoff's law cannot be satisfied at the vertex $v$. Hence, $\Gamma$ with this embedding does not have a nowhere-identity dihedral $3$-flow.\\
\indent Let $f: E \rightarrow D_6$ be given by
\begin{equation}\label{equation:flowcounterexa}
f(e) = \left\{\begin{array}{ll} r^2 & \text{if $e \in \{e_1,e_2,e_3\}$},\\ r^2s & \text{if $e \in \{e_4,e_6,e_8\}$},\\ rs & \text{if $e \in \{e_5,e_7,e_9\}$},\end{array}\right.
\end{equation}
with $r$ and $s$ as in equation (\ref{equation:dihedraal}). Then $f$ is a nowhere-identity $D_6$-flow with respect to the orientation as given in Figure \ref{fig:M1}. Observe that the loops are assigned reflections and hence do not require an orientation. We conclude that statement (\ref{equation:main}) does not hold for $n=3$.\\
\indent We will now generalize the counterexample in Figure \ref{fig:M1}. Call a bridge of an embedded graph $\Gamma$ \textit{plane-sided} if its removal results in an embedded graph at least one of whose components has genus zero (the embedding of components being inherited from that of $\Gamma$). The next lemma shows that plane-sided bridges form an obstruction for the existence of nowhere-identity $G$-flows on an embedded graph.

\begin{lemma}\label{lemma:planesided}
If an embedded graph has a plane-sided bridge, then it does not have a nowhere-identity $G$-flow, for any group $G$.
\end{lemma}
\begin{proof}
Let $e = uv$ be a plane-sided bridge of an embedded graph $\Gamma = (V,E)$, with $u,v \in V$. Let $\Gamma^{\circ}$ denote a component of $\Gamma\setminus e$ of genus zero. Assume that $\Gamma^{\circ}$ contains $v$. Then we can modify $\Gamma$ by contracting edges of $\Gamma^{\circ}$ so that it results in an embedded graph $\Gamma'$ in which the edges incident with $v$ are $e$ and a (possibly empty) collection of non-intersecting loops. Any nowhere-identity $G$-flow in $\Gamma$ gives rise to a nowhere-identity $G$-flow in $\Gamma'$ and vice versa. But in $\Gamma'$ any $G$-flow must assign the identity to $e$, proving the lemma.
\end{proof}

\noindent In fact, the following theorem of DeVos \cite{devos00} (but taken as it is stated in \cite{goodall16}) shows that if the commutator subgroup of $G$ is large enough, then a plane-sided bridge is the only obstruction. In the theorem, $Q_8$ denotes the quaternion group. 

\begin{theorem}\label{theorem:devos}
Let $\Gamma$ be an embedded graph and $G$ a finite nonabelian group.
\begin{enumerate}
\item If $|[G,G]| > 2$, then $\Gamma$ admits a nowhere-identity $G$-flow if and only if $\Gamma$ does not have a plane-sided bridge.
\item If $|[G,G]| = 2$ and $G \notin \{D_8,Q_8\}$, then $\Gamma$ has a nowhere-identity $G$-flow if and only if there does not exist an odd-sized set $B$ of bridges of $\Gamma$ such that each $e \in B$ is a plane-sided bridge of $\Gamma \setminus (B\setminus e)$.
\item If $G \in \{D_8, Q_8\}$, then $\Gamma$ has a nowhere-identity $G$-flow in case $\Gamma$ has no bridge. It is \text{NP}-complete to decide if $\Gamma$ has a nowhere-identity $G$-flow if there is a bridge. 
\end{enumerate}
\end{theorem}
\indent Suppose that a graph $\Gamma = (V,E)$ with rotation system $\pi$ has a vertex $v$ of degree $d$, with $d$ odd, and such that $\Gamma\setminus v$ has $d$ connected components. Then there is no way of assigning the non-identity elements of the set in equation (\ref{equation:intersection}) to the $d$ edges incident with $v$, such that the product (in any order) equals the identity. In fact, to $v$ we can add any number of non-intersecting loops and we can decontract edges (just as in the proof of Lemma \ref{lemma:planesided}). The existence of a vertex $v$ as described above, is then seen to correspond to the existence of a subset $X \subseteq V$ such that $\delta(X)$ consists entirely of bridges, is odd-sized and such that the subgraph induced by $X$ is planar with respect to $\pi$. Such a $\delta(X)$ corresponds bijectively with a set $B$ as described in the second part of Theorem \ref{theorem:devos}. Combining the above with the first two items of Theorem \ref{theorem:devos}, the following is derived. 

\begin{theorem}\label{theorem:obstrd6}
Let $\Gamma$ be an embedded graph. Then $\Gamma$ has a nowhere-identity $D_6$-flow but no nowhere-identity dihedral $3$-flow if $\Gamma$ has no plane-sided bridge but there exists an odd-sized subset $B$ of bridges of $\Gamma$ such that each $e \in B$ is a plane-sided bridge of $\Gamma \setminus (B \setminus e)$. 
\end{theorem}

\noindent As $[\D,\D] \cap \D^{<3} = [\D,\D] \cap \D^{<4}$, the presence of a set of bridges as in Theorem \ref{theorem:obstrd6} also is an obstruction for the existence of nowhere-identity dihedral $4$-flows. 

\subsection{The case $n=2$}\label{subsection:nis2}

The dihedral group $D_{2n}$ is usually defined for $n \geq 3$. However, the description of $D_{2n}$ in terms of the group presentation as given just above equation (\ref{equation:dihedraal}), as well as the definitions of $\D^{<n}$ and $\Pi_n$ and the statement in (\ref{equation:main}), allow for an investigation for the case $n=2$. As $D_4 \cong V_4$, the Klein four-group, Theorem \ref{theorem:abelian} shows that the existence of a nowhere-identity $D_4$-flow is equivalent with the existence of a nowhere-zero $4$-flow. In this section we will restrict attention to cubic graphs. For cubic graphs, specifying a nowhere-zero $4$-flow is the same as specifying a $3$-edge coloring \cite[Prop.\ $6.4.5(\text{ii})$]{diestel00}. The following example shows that not all cubic graphs with a $3$-edge coloring also have a nowhere-identity dihedral $2$-flow (with respect to a given embedding).

\begin{examp}\label{examp:2punt}
Consider the (unique) embedding in the torus of the cubic graph $\Gamma = (V,E)$ on two vertices with three parallel edges. Write $V = \{v_1,v_2\}$ and $E = \{e_1, e_2,e_3\}$. Then the rotation system is given by $\pi_{v_1} = \pi_{v_2} = (e_1e_2e_3)$. Let $f$ be a nowhere-identity dihedral $2$-flow. Then there must be a unique rotation, say $f(e_1)$. The flow equations read $f(e_1)f(e_2)f(e_3) = f(e_1)^{-1}f(e_2)f(e_3) = 1$ (as reflections equal their own inverse). Hence $f(e_1)^2 = 1$, the identity of $\D^{<2}$. This implies that $f(e_1) = 1$, contradicting the fact that $f$ is nowhere-identity.
\end{examp}

\indent We will characterize when an embedded cubic graph admits a nowhere-identity dihedral $2$-flow. Let us first consider a nowhere-identity $D_{2n}$-flow $f$ on an embedded graph $\Gamma = (V,E)$ (not necessarily cubic). For every $v \in V$, the number of edges incident with $v$ on which $f$ evaluates to a reflection is even. Assume now that $\Gamma$ is cubic. For each $0 \leq a < n$, if at any vertex one of the edges incident with it is assigned a reflection of the form $r^as$, with $r$ and $s$ as in equation (\ref{equation:dihedraal}), then this is the only such edge. For $0 \leq a < n$, define
\[
M_a := \{e \in E \mid f(e) = r^as\}.
\]
\noindent We summarize the above observations in a lemma.

\begin{lemma}\label{lemma:regularmatching}
Let $f$ be a nowhere-identity $D_{2n}$-flow on an embedded cubic graph. Then the subgraph consisting of edges that are reflections is $2$-regular. For each $0 \leq a < n$, the set $M_a$ is a matching.
\end{lemma}

Since a nowhere-identity $D_{2n}$-flow can consist entirely of rotations, there is no guarantee in general how large the matchings $M_a$ can be. For $n=2$ however, $M_0$ and $M_1$ are perfect matchings (and hence so is $E \setminus (M_0 \cup M_1)$). We examine what this means for $\D^{<2}$-valued flows. Write $\D^{<2} = \{1,x,x^{-1},y,z,w\}$ with
\[
1 = \begin{pmatrix} 1 & 0 \\ 0 & 1 \end{pmatrix}, x = \begin{pmatrix} 1 & 1 \\ 0 & 1 \end{pmatrix}, y = \begin{pmatrix} -1 & 1 \\ 0 & 1 \end{pmatrix}, z = \begin{pmatrix} -1 & 0 \\ 0 & 1 \end{pmatrix} \text{ and } w = \begin{pmatrix} -1 & -1 \\ 0 & 1 \end{pmatrix}.
\]
If $\pi_{\text{ab}}: \D \rightarrow \D/[\D,\D] \cong D_4$ denotes the canonical homomorphism, then for every $g \in D_4$ it holds that $\pi_{\text{ab}}^{-1}(\{g\}) \cap \D^{<2} = \Pi_{2}^{-1}(\{g\}) \cap \D^{<2}$, where $\Pi_2$ is the map from (\ref{equation:pin}). Therefore, as $M_0$ and $M_1$ are perfect matchings for every nowhere-identity $D_4$-flow, it follows that for every nowhere-identity dihedral $2$-flow $f$, the partition
\begin{equation}\label{equation:1factor}
E = \{e \mid f(e) = z\} \sqcup \{e \mid f(e) \in \{x^{\pm1}\}\} \sqcup \{e \mid f(e) \in \{y,w\}\},
\end{equation}
is a $1$-factorization, i.e., a $3$-edge coloring. We now prove the following theorem.

\begin{theorem}\label{theorem:equivalentie}
A cubic graph has a $3$-edge coloring if and only if it admits a nowhere-identity dihedral $2$-flow with respect some embedding.
\end{theorem}
\begin{proof}
Let $\Gamma = (V,E)$ be a cubic graph. By equation (\ref{equation:1factor}) every nowhere-identity dihedral $2$-flow of an embedding of $\Gamma$ yields a $3$-edge coloring of the underlying graph, showing the ``if" statement.\\
\indent To see ``only if", suppose $c: E \rightarrow \{1,2,3\}$ is a $3$-edge coloring of $\Gamma$. Direct the edges $e$ for which $c(e) = 3$ arbitrarily. For $v \in V$, set the cyclic order $\pi_v$ at $v$ as follows
\[
\pi_v = \left\{\begin{array}{ll} (e_1e_2e_3) & \text{if $e_3$ is incoming,}\\ (e_1e_3e_2) & \text{if $e_3$ is outgoing,}\end{array}\right.
\]
where $e_i$ is such that $c(e_i) = i$, for $i = 1,2,3$. Define $f: E \rightarrow \D^{<2}$ by 
\[
f(e) =  \left\{\begin{array}{ll} z & \text{if $c(e) = 1$}, \\ y & \text{if $c(e) = 2$}, \\ x & \text{if $c(e) = 3$}.\end{array}\right.
\]
It is a straightforward calculation to verify that $f$ is a nowhere-identity dihedral $2$-flow with respect to the embedding of $\Gamma$ given by $\{\pi_v \mid v \in V\}$.
\end{proof}

\begin{corollary}\label{corollary:npcomplete}
It is $\text{NP}$-complete to decide if a cubic graph has a nowhere-identity dihedral $2$-flow for some embedding.
\end{corollary}
\begin{proof}
A nowhere-identity dihedral $2$-flow constitutes a certificate, hence the problem is in $\text{NP}$. Theorem \ref{theorem:equivalentie} shows that it is $\text{NP}$-complete.
\end{proof}

The proof of Theorem \ref{theorem:equivalentie} shows that for every $3$-edge colorable cubic graph there exists a nowhere-identity dihedral $2$-flow with respect to some embedding. It is not true that a $3$-edge colorable cubic admits such a flow for every embedding (Example \ref{examp:2punt}). The proof of Theorem \ref{theorem:equivalentie} also can be used to characterize nowhere-identity dihedral $2$-flows on embedded cubic graphs. Let $\Gamma = (V,E)$ be a cubic graph with rotation system $\pi$. Let $f: E \rightarrow \{1,2,3,4\}$ be a $4$-edge coloring. If the sets $\{e \mid f(e) = 1\}$ and $\{e \mid f(e) = 3\}$ form a perfect matching and if locally, with respect to $\pi$, every two vertices that share an edge $e$ for which $f(e) = 3$, look like one of
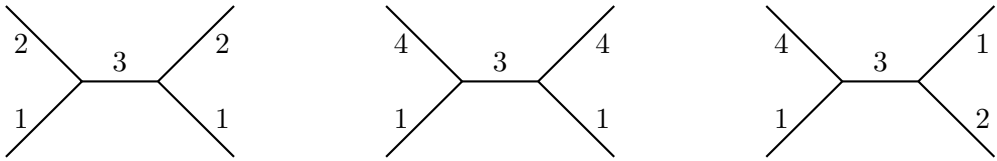
\begin{figure}[H]
\centering
\begin{tikzpicture}
\draw[thick]
    (0,0) -- (1,1) node[pos=0.5]{\hspace{-6mm}1}
 -- (0,2) node[pos=0.5]{\hspace{-6mm}2};
\draw[thick]
 (1,1) --(2,1)node[pos=0.5,anchor=south]{3}
--(3,0)node[pos=0.5]{\hspace{7mm}1};
\draw[thick]
(2,1)--(3,2)node[pos=0.5]{\hspace{7mm}2};

\draw[thick]
    (5,0) -- (6,1) node[pos=0.5]{\hspace{-6mm}1}
 -- (5,2) node[pos=0.5]{\hspace{-6mm}4};
\draw[thick]
 (6,1) --(7,1)node[pos=0.5,anchor=south]{3}
--(8,0)node[pos=0.5]{\hspace{7mm}1};
\draw[thick]
(7,1)--(8,2)node[pos=0.5]{\hspace{7mm}4};

\draw[thick]
    (10,0) -- (11,1) node[pos=0.5]{\hspace{-6mm}1}
 -- (10,2) node[pos=0.5]{\hspace{-6mm}4};
\draw[thick]
 (11,1) --(12,1)node[pos=0.5,anchor=south]{3}
--(13,0)node[pos=0.5]{\hspace{7mm}2};
\draw[thick]
(12,1)--(13,2)node[pos=0.5]{\hspace{7mm}1};
\end{tikzpicture}
\caption{special $4$-edge coloring}\label{fig:4ec}
\end{figure}
\noindent then $f$ is called a \textit{special} $4$\textit{-edge coloring} (with respect to $\pi$). 

\begin{theorem}
An embedded cubic graph has a nowhere-identity dihedral $2$-flow if and only if it has a special $4$-edge coloring.
\end{theorem}
\begin{proof}
Consider the equation $abc = 1$ in $\D^{<2}$, with variables $a,b$ and $c$. If none of the variables is allowed to be the identity, then, fixing $a = z$, the set of ordered pairs $(b,c)$ of solutions to $abc=1$ is given by 
\[
\{(x,w),(x^{-1},y),(y,x),(w,x^{-1})\}.
\]
The correspondence $x^{\pm 1} \mapsto 3, y \mapsto 4, w \mapsto 2$ and $z \mapsto 1$ yields the desired result. Indeed, note that in any of the pictures in Figure \ref{fig:4ec}, there is a unique way to direct the middle edge so that the correspondence just described gives a flow.
\end{proof}

\noindent We conclude the section with a question. 

\begin{question}
What is the complexity of deciding whether a given embedded cubic graph has a special $4$-edge coloring?
\end{question}

\subsection{Counting}\label{subsection:counting}

In this section $G$ denotes an abelian group of order $n$. In the introduction it was mentioned that the number of nowhere-zero $n$-flows in a graph is at least the number of nowhere-zero $G$-flows. The next example exhibits an embedded graph that admits nowhere-identity dihedral $4$-flows, but there are fewer of these than nowhere-identity $D_{8}$-flows.

\begin{examp}\label{examp:6vertextorus}
Let $\Gamma = (V,E)$ be the graph depicted in Figure \ref{fig:tellen}. There are four ways to assign a cyclic order to the leftmost three vertices, such that there is only one face containing those vertices. By symmetry, there are ten different rotation systems with respect to which $\Gamma$ has one face. By Euler's formula (\ref{equation:euler}), these all are embeddings of $\Gamma$ in the $2$-torus. By computer-aided calculation, each of these ten embeddings of $\Gamma$ in the $2$-torus has $576$ nowhere-identity $D_{8}$-flows and only $512$ nowhere-identity dihedral $4$-flows.
\end{examp}

\begin{figure}[H]
\centering
\begin{tikzpicture}
\SetGraphUnit{1.3}
\Vertex{a}
\EA(a){b}
\EA(b){c}
\EA(c){d}
\EA(d){e}
\EA(e){f}
\Edge[](a)(b)
\Edge[style={bend left=90}](a)(b)
\Edge[](b)(c)
\Edge[](c)(d)
\Edge[](d)(e)
\Edge[](e)(f)
\Edge[style={bend left=30}](a)(c)
\Edge[style={bend left=90}](e)(f)
\Edge[style={bend left=30}](d)(f)
\end{tikzpicture}
\caption{A cubic graph embeddable in the $2$-torus}\label{fig:tellen}
\end{figure}
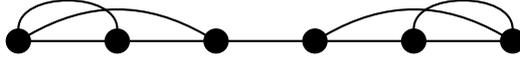

It was also mentioned in the introduction that both the number of nowhere-zero $n$-flows and the number of nowhere-zero $G$-flows are counted by a polynomial in $n$ \cite{kochol02,tutte54}. Example $5.11$ of \cite{goodall16} demonstrates that the number of nowhere-identity $D_{2n}$-flows in an embedded graph is a quasi-polynomial in $n$ of period $2$. The authors of \cite{goodall16} asked if the number of nowhere-identity dihedral $n$-flows is also counted by a quasi-polynomial in $n$ of period $2$. We were neither able to confirm nor to contradict this.

\section{Positive results}\label{section:pos}

In this section we discuss cases in which statement (\ref{equation:main}) holds. As in the previous section, all embedded graphs are embedded in an orientable surface. In addition, every graph in this section is assumed to be cubic.

\subsection{Non-zero reflection cycles}\label{subsection:nonzero}

In this section and the following we fix a cubic graph $\Gamma = (V,E)$ embedded in a closed orientable surface $S$. Let $f$ be a nowhere-identity $D_{2n}$-flow of $\Gamma$. If $f$ never evaluates to a reflection, then $f$ can be viewed as a nowhere-zero $\Z_n$-flow. By Theorem \ref{theorem:abelian}, $\Gamma$ then admits a nowhere-zero $n$-flow, which in particular is a nowhere-identity dihedral $n$-flow. We assume now that $f$ also evaluates to reflections.\\
\indent By Lemma \ref{lemma:regularmatching} the subgraph $\Gamma_f$ of $\Gamma$ consisting of edges that are reflections, is $2$-regular. A connected component of $\Gamma_f$ is called a \textit{reflection cycle with respect to $f$}. If the flow is clear from the context, we just speak of reflection cycles. For $a \in \Z$ and a reflection cycle $C \subseteq \Gamma_f$, define
\begin{equation}\label{equation:aC}
f_{a,C}: E \rightarrow D_{2n}, \hspace{2mm} e \mapsto \left\{\begin{array}{ll} f(e)r^{-a} & \text{if }e \in C, \\ f(e) & \text{otherwise,}\end{array}\right.
\end{equation}
with $r$ as in equation~(\ref{equation:dihedraal}). The map $f_{a,C}$ does not evaluate to the identity as by assumption $f(e)$ is a reflection for each $e \in C$. For all reflections $x,y \in D_{2n}$ we have $xr^{-a}yr^{-a} = xr^{-a}r^{a}y = xy$. These observations prove the following lemma.
\begin{lemma}\label{lemma:nietnulmaken}
Let $f$ be a nowhere-identity $D_{2n}$-flow of $\Gamma$ embedded in $S$, let $C$ be a reflection cycle and let $a \in \Z$. Then the map $f_{a,C}$ as defined in equation (\ref{equation:aC}) is again a nowhere-identity $D_{2n}$-flow.
\end{lemma}

A reflection cycle with no edge on which $f$ evaluates to $s$, as in equation (\ref{equation:dihedraal}), is called a \textit{non-zero reflection cycle}. When a reflection cycle $C$ of a $D_{2n}$-flow $f$ does not contain all $n$ reflections among its edge values, a value of $a \in \Z$ can be chosen so that the same cycle $C$ in the flow $f_{a,C}$ is a non-zero reflection cycle. In a moment we will see why non-zero reflection cycles are convenient. First we need another topological concept.\\
\indent Recall that $S$ denotes the surface in which $\Gamma$ is embedded. A loop in $S$ (i.e., a continuous image of the circle in $S$) is \textit{contractible} if it is homotopic to a point (i.e., the constant mapping). A loop is \textit{simple} if it has no self-intersections. \\
\indent Let $C$ be a simple contractible cycle of $\Gamma$ embedded in $S$. Define
\begin{equation}\label{equation:C}
f_C: E \rightarrow D_{2n}, \hspace{2mm} e \mapsto \left\{\begin{array}{ll} f(e)s & \text{if $e \in C$},\\ f(e) & \text{otherwise.}\end{array}\right.
\end{equation}
As $C$ is simple and contractible, it bounds a disk $D$ in $S$. The interior of $D$ is denoted by $D^{\circ}$. If $g: E \rightarrow D_{2n}$ is a map, then for $v \in V$ the product of the values on the edges incident with $v$ (in the order $\pi_v$) is denoted by $P_{g,\pi}(v)$, or by $P_{g}(v)$, if the rotation system is understood. The following lemma is of crucial importance in the remainder of Section \ref{section:pos}. 

\begin{lemma}\label{lemma:vermenigmets}
Let $f$ be a $D_{2n}$-flow of the cubic graph $\Gamma$ embedded in the orientable surface $S$. Let $C$ be a simple cycle of $\Gamma$ contractible in $S$, bounding a disk $D$ in $S$. Assume that $f$ does not evaluate to reflections on edges in the interior $D^{\circ}$ of $D$. Then the map $f_C$ as defined in equation (\ref{equation:C}) is again a $D_{2n}$-flow.
\end{lemma}
\begin{proof}
The disk $D$ can be drawn in the plane such that all vertices on $C$ have a consistent cyclic order (all clockwise or all counterclockwise in $S$). Choose a local orientation of the cycle $C$ opposite to the sense of these cyclic orders (counterclockwise if the sense of the cyclic orders is clockwise, or vice versa). We may assume that all edges on $C$ on which $f$ evaluates to rotations, are directed in accordance with the local orientation of $C$. \\
\indent In order to prove the lemma, we first specify an orientation of the edges on which $f_C$ is equal to a rotation. Edges in $D^{\circ}$, on which $f_C$ coincides with $f$, are by assumption given a rotation value by $f$, and these edges have their orientations reversed. Edges in $D\setminus D^{\circ}$, on which $f_C$ also coincides with $f$, retain their original orientation. Any edges that are given a rotation value by $f_C$ but not by $f$ lie on $C$: the orientation for these edges is chosen in accordance with the local orientation of $C$. That $f_C$ is a flow now follows from the observation that for $v \in V$ we have
\[
P_{f_C}(v) = \left\{\begin{array}{ll} P_{f}(v)^{-1} = 1 & \text{if $v \in D_{\text{in}} \cup D_{\text{out}}$},\\P_f(v) = 1 & \text{otherwise,}\end{array}\right.
\]
where $D_{\text{in}}$ consists of those vertices in $D$ that are incident with an edge that lies in $D^{\circ}$, and $D_{\text{out}}$ consists of vertices on $C$ that are incident with an edge outside $D$ that is a reflection. 
\end{proof}

\noindent In fact, the proof given above shows that Lemma \ref{lemma:vermenigmets} still holds when in the statement ``$D_{2n}$-flow'' is replaced by ``dihedral $n$-flow'', while replacing the element $s$ in the definition (\ref{equation:C}) of $f_C$ by $\begin{psmallmatrix} -1 & 0 \\ 0 & 1 \end{psmallmatrix} \in \D^{<n}$. This will be needed in Section \ref{subsection:snarks}. 

\begin{theorem}\label{theorem:nonzerorc}
Let $f$ be a nowhere-identity $D_{2n}$-flow of $\Gamma$ embedded in $S$. If all reflection cycles are non-zero and contractible (on the surface $S$), then $\Gamma$ admits a nowhere-zero $n$-flow.
\end{theorem}
\begin{proof}
The proof goes via induction on the number $t$ of non-zero contractible reflection cycles. If $t = 0$ then we are done by the remark at the beginning of this section. Suppose that $t > 0$. Let $C$ be a non-zero reflection cycle. As it is a reflection cycle, it is simple. We may assume that all edges in the interior of the disk that $C$ bounds, are rotations (otherwise we would have continued with a non-zero reflection cycle enclosed by $C$). Then $f_C$ is a $D_{2n}$-flow by Lemma \ref{lemma:vermenigmets}. Moreover, it is nowhere-identity as $C$ is a non-zero reflection cycle. By construction we have produced a nowhere-identity $D_{2n}$-flow with $t-1$ non-zero contractible reflection cycles, hence by the induction hypothesis we are done.
\end{proof}

\indent Neither of the two assumptions in the statement of Theorem \ref{theorem:nonzerorc} can be weakened. Indeed, the nowhere-identity $D_6$-flow $f$ described in equation~(\ref{equation:flowcounterexa}) of the embedded graph in Figure~\ref{fig:M1} has the property that all reflection cycles are non-zero and non-contractible. Since the embedded graph in Figure~\ref{fig:M1} has bridges, it does not have a nowhere-zero $3$-flow. On the other hand, it is easy to construct a nowhere-zero $D_6$-flow of a planar embedding of the complete graph $K_4$ on four vertices, with a reflection cycle that is not non-zero and which is contractible (as is any cycle in the plane). Since $K_4$ is not bipartite, it does not have a nowhere-zero $3$-flow (\cite[Prop.\ $6.4.2$]{diestel00}).\\
\indent Lemma \ref{lemma:nietnulmaken} may be useful in creating non-zero reflection cycles, for instance if $n$ is larger than the size of the largest reflection cycle. We remark that contractibility of cycles can be tested in linear time \cite{cabello11}. In Section \ref{subsection:short} we give applications of Theorem \ref{theorem:nonzerorc}. 

\subsection{Short reflection cycles}\label{subsection:short}

Given a nowhere-identity $D_{2n}$-flow $f$, call a nowhere-identity dihedral $n$-flow $\widetilde{f}$ for which $\Pi_n \circ \widetilde{f} = f$ (with $\Pi_n$ as in equation (\ref{equation:pin})) a \textit{lift} of $f$. In Theorem \ref{theorem:nonzerorc} we have seen that non-zero contractible reflection cycles can be modified to rotations to ensure the existence of a lift. Example \ref{examp:6vertextorus} demonstrates that this does not happen always. We now consider lifts of flows that have non-contractible reflection cycles. We will show that for odd $n$, unique reflection cycles that are short can be adjusted to rotations, in which case the given nowhere-identity $D_{2n}$-flow can be lifted to a nowhere-identity dihedral $n$-flow taking rotation values only, equivalently, a nowhere-zero $n$-flow. It is worth examining an example first.

\begin{examp}\label{examp:2punt2}
Consider the graph from Example \ref{examp:2punt}. It admits nowhere-identity $D_{2n}$-flows\footnote{Indeed, the assignment $f(e_1) =  r^{n/2}, f(e_2) = r^{n/2}s$ and $f(e_3) = s$ specifies a flow $f$ for even $n$. For odd $n$, one can take $f(e_1) = r^{n-2}, f(e_2) = r$ and $f(e_3) = r$.} for all $n \geq 2$. If $f$ is a nowhere-identity $D_{2n}$-flow with reflections, then $n$ is even. Consider such a flow. Then the reflection cycle is non-contractible. By the same argument as in Example \ref{examp:2punt}, any lift of $f$ is forced to assign the identity to the edge $e_1$. Hence, any nowhere-identity dihedral $n$-flow (which exists for $n \geq 3$) consists entirely of rotations.\end{examp}

Recall that we have fixed a cubic graph $\Gamma = (V,E)$ embedded in a closed orientable surface $S$.

\begin{proposition}\label{proposition:lengte2}
Let $f$ be a nowhere-identity $D_{2n}$-flow of $\Gamma$ embedded in $S$, with $n$ odd. If there is a unique reflection cycle and it is of length $2$, then $\Gamma$ admits a nowhere-zero $n$-flow. 
\end{proposition}
\begin{proof}
Let $C$ denote the unique reflection cycle. After possibly considering the flows $f_{a,C}$ from Lemma \ref{lemma:nietnulmaken} for some $a \in \Z$, we may assume that $C$ is a non-zero reflection cycle.  We will argue that either $C$ is contractible or $\Gamma$ is the graph from Example \ref{examp:2punt}. For the latter case it is clear that $\Gamma$ admits nowhere-zero $n$-flows. In case $C$ is contractible, then applying Theorem \ref{theorem:nonzerorc} yields the proposition.\\
\indent Suppose that $C$ is non-contractible and that $\Gamma$ is not the graph from Example \ref{examp:2punt}. Let $v_1$ and $v_2$ denote the vertices on $C$. As $C$ is non-contractible, we may assume that the cyclic orders at $v_1$ and $v_2$ are given by $\pi_{v_1} = (e_1e_2e_3)$ and $\pi_{v_2} = (e_2e_3e_4)$ (where $e_1 \neq e_4$). Write $f(e_1) = r^a, f(e_2) = r^bs, f(e_3) = r^cs$ and $f(e_4) = r^d$, with $0 < a,b,c,d < n$. Assuming that $e_1$ and $e_4$ are directed towards $v_1$ and $v_2$, respectively, the flow equations yield $a+b-c \equiv 0 \text{ mod }n$ and $d+b-c \equiv 0 \text{ mod }n$. It follows that $a \equiv d \text{ mod }n$. On the other hand, on the cut-set $X := \{v_1,v_2\}$ we calculate that
\[
r^{a+d} = \prod_{e \in \delta(X)}f(e) = \prod_{e \in \delta(V\setminus X)}f(e) = 1,
\]
where the second and third equality follow from the fact that $C$ is the unique reflection cycle (i.e., the rest of the graph consists of rotations). Hence we also must have $a \equiv -d \text{ mod }n$, which is impossible for odd $n$. 
\end{proof}

A result in the same vein holds for reflection cycles of length $3$. We first need a lemma.

\begin{lemma}\label{lemma:bridgerc}
If the graph $\Gamma$ has a bridge, then there cannot exist a nowhere-identity $D_{2n}$-flow of $\Gamma$ embedded in $S$ with a unique reflection cycle.
\end{lemma}
\begin{proof}
Let $e$ be a bridge of $\Gamma$. Assume there exists a nowhere-identity $D_{2n}$-flow $f$ with unique reflection cycle $C$. Let $\Gamma_e$ denote the connected component of $\Gamma \setminus e$ not containing $C$. Then $\Gamma_e$ consists entirely of edges on which $f$ evaluates to rotations. As rotations commute, we can contract edges in $\Gamma_e$ and modify cyclic orders at vertices in $\Gamma_e$ to obtain a planar graph with nowhere-identity $D_{2n}$-flow $f'$. But then $e$ is plane-sided and the existence of $f'$ contradicts Lemma \ref{lemma:planesided}.
\end{proof}

\begin{proposition}\label{proposition:lengte3}
Let $f$ be a nowhere-identity $D_{2n}$-flow of $\Gamma$ embedded in $S$, with $n$ odd and $n \geq 5$. If there is a unique reflection cycle and it is of length $3$, then $\Gamma$ admits a nowhere-zero $n$-flow.
\end{proposition}
\begin{proof}
Without loss of generality, we may assume that the unique reflection cycle $C$ is non-zero. We show that $C$ is contractible. If this is true, then Theorem \ref{theorem:nonzerorc} settles the proposition. Assume, to the contrary, that $C$ is non-contractible. Let $v_1,v_2$ and $v_3$ denote the vertices on $C$. If there are parallel edges between $v_1,v_2$ and $v_3$, then $\Gamma$ has a bridge $e$. But then $f$ cannot exist by Lemma \ref{lemma:bridgerc}.\\
\indent Write $e_1 = v_1v_2, e_2 = v_2v_3$ and $e_3 = v_3v_1$. For $i = 1,2,3$, let $e_{3+i}$ be the remaining edge incident with $v_i$. The above shows that $e_{3+i} = e_{3+j}$ if and only if $i = j$. As $C$ is non-contractible, the cyclic orders may be taken as follows 
\[
\pi_{v_{1}} = (e_1e_3e_4), \pi_{v_2} = (e_1e_2e_5) \text{ and }\pi_{v_3} = (e_3e_2e_6).
\]
Let $f(e_{3+i}) = r^{a_i}$, with $0 < a_i < n$ and assume that $e_{3+i}$ is directed towards $v_i$, for $i = 1,2,3$. Similarly as in the proof of Proposition \ref{proposition:lengte2}, the Kirchhoff equations imply that $a_4+a_6 \equiv a_5 \text{ mod }n$. However, the cut-set condition implies that $a_4 + a_6 \equiv -a_5 \text{ mod }n$. In conclusion, such an $f$ cannot exist for odd $n$.
\end{proof}

\noindent In view of Seymour's $6$-flow theorem \cite{seymour81}, Proposition \ref{proposition:lengte3} is only interesting for the case $n = 5$ for bridgeless cubic graphs.\\
\indent Having discussed unique reflection cycles of lengths $2$ and $3$, it is natural to consider longer reflection cycles. Although something can be said about the case of a unique reflection cycle having length $4$, we do not include the result here, as there are many case distinctions and calculations involved. In the next section, we restrict attention to snarks. As snarks have girth at least $5$, certainly none of the before-mentioned propositions apply.

\subsection{Snarks}\label{subsection:snarks}

A \textit{snark} is a simple, connected, bridgeless cubic graph that is not $3$-edge colorable (i.e., without a nowhere-zero $4$-flow) and with girth at least $5$. Snarks are interesting because they are potential (minimal) counterexamples to many important conjectures, e.g., Tutte's $5$-flow conjecture \cite{tutte54} and the cycle double cover conjecture \cite{seymour79,szekeres73}. Any nowhere-identity $D_8$-flow on a snark necessarily has a reflection cycle. Hence, the question that we will address in this section is whether snarks admit nowhere-identity dihedral $4$-flows. The following theorem is derived from Lemma \ref{lemma:vermenigmets}, and is due to Lex Schrijver.

\begin{theorem}\label{theorem:removal}
Let $\Gamma = (V,E)$ be a cubic graph. If there exists an $e' \in E$ such that $\Gamma \setminus e'$ has a nowhere-zero $n$-flow, then $\Gamma$ has a nowhere-identity dihedral $n$-flow for any embedding for which $e'$ is contained in a simple contractible cycle. 
\end{theorem}
\begin{proof}
Let $f$ be a nowhere-zero $n$-flow on $\Gamma \setminus e'$. Consider an embedding of $\Gamma$ for which $e'$ is contained in a simple contractible cycle $C$. Let $f'$ be the $n$-flow on $\Gamma$ that agrees with $f$ on $\Gamma \setminus e'$ and is zero on $e'$. By the remark following Lemma \ref{lemma:vermenigmets}, $f'_C$ is a dihedral $n$-flow and is immediately seen to be nowhere-identity.
\end{proof}

\noindent There also exists an edge-contraction version of Theorem \ref{theorem:removal}, in which the deletion $\Gamma \setminus e'$ is replaced by the contraction $\Gamma/e'$. As its proof follows the same line as the proof given above, we do not include it here. In the first application of Theorem \ref{theorem:removal} below, we use the fact that a facial walk in an (cellullarly) embedded graph is a contractible cycle.\\
\indent The smallest snark is the ubiquitous Petersen graph. The (orientable) surfaces in which it can be embedded are the tori with one, two and three holes. In all cases, the Petersen graph has nowhere-identity $D_{2n}$-flows, for $n \geq 3$, by Theorem \ref{theorem:devos}. We prove that it also has nowhere-identity dihedral $4$-flows. In the following, we say a graph $\Gamma$ has a nowhere-identity dihedral $4$-flow if for each orientable surface $S$ in which $\Gamma$ can be embedded there is some embedding of $\Gamma$ in $S$ which has a nowhere-identity dihedral $4$-flow. 

\begin{proposition}\label{proposition:petersen}
The Petersen graph has a nowhere-identity dihedral $4$-flow.
\end{proposition}
\begin{proof}
Consider the (usual) drawing of the Petersen graph in the plane as a pentagon with an inner pentagram that has five spokes. Let all vertices be given a consistent cyclic order (e.g., clockwise). Then the face tracing algorithm ($\S 3.2.6$ of \cite{gross87}) yields three faces and the graph is embedded in the $2$-torus by Euler's formula (\ref{equation:euler}). Removing any edge from the face that corresponds to the walk along the outer pentagram, which is a simple facial walk, results in a subdivision $\Gamma'$ of a cubic graph $\Gamma$ on eight vertices. As the Petersen graph is the smallest snark, $\Gamma$ has a nowhere-zero $4$-flow, hence so does $\Gamma'$. By Theorem \ref{theorem:removal}, this embedding of the Petersen graph has a nowhere-identity dihedral $4$-flow.\\
\indent We consider the same rotation system as above, but reverse the cyclic order at two vertices on the outer pentagram that are a distance two apart. Then there are five faces and hence the Petersen graph is embedded in the $1$-torus. The five spokes of the inner pentagram lie on a simple facial walk of length five. By the same argument as before, this embedded Petersen graph permits a nowhere-identity dihedral $4$-flow.\\
\indent Lastly, there exists an embedding of the Petersen graph in the $3$-torus. In this case, there is only one face and hence we cannot apply Theorem \ref{theorem:removal} directly. We give a nowhere-identity dihedral $3$-flow. Consider the drawing of the Petersen graph in Figure \ref{fig:Petersen}.\\

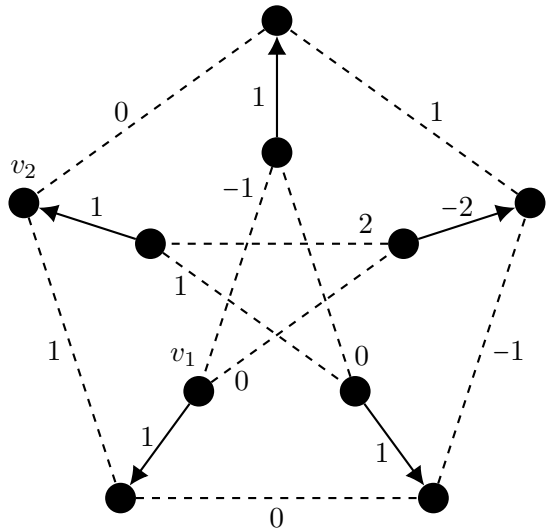
\begin{figure}[H]
\centering
\begin{tikzpicture}
[scale=0.7,vertex_style/.style={draw,circle,very thick,fill=black}]
 
\useasboundingbox (-5.05,-4.4) rectangle (5.1,5.25);
 
\begin{scope}[rotate=90]
\node[label={},style={draw,circle,very thick,fill=black}] (1) at (canvas polar cs: radius=2.5cm, angle=0){};
\node[label={},style={draw,circle,very thick,fill=black}] (2) at (canvas polar cs: radius=2.5cm, angle=72){};
\node[label={\hspace{-4mm}$v_1$},style={draw,circle,very thick,fill=black}] (3) at (canvas polar cs: radius=2.5cm, angle=144){};
\node[label={},style={draw,circle,very thick,fill=black}] (4) at (canvas polar cs: radius=2.5cm, angle=216){};
\node[label={},style={draw,circle,very thick,fill=black}] (5) at (canvas polar cs: radius=2.5cm, angle=288){};

\node[label={},style={draw,circle,very thick,fill=black}] (6) at (canvas polar cs: radius=5cm, angle=0){};
\node[label={\hspace{0mm}$v_2$},style={draw,circle,very thick,fill=black}] (7) at (canvas polar cs: radius=5cm, angle=72){};
\node[label={},style={draw,circle,very thick,fill=black}] (8) at (canvas polar cs: radius=5cm, angle=144){};
\node[label={},style={draw,circle,very thick,fill=black}] (9) at (canvas polar cs: radius=5cm, angle=216){};
\node[label={},style={draw,circle,very thick,fill=black}] (10) at (canvas polar cs: radius=5cm, angle=288){};

\end{scope}

\draw[style={->, thick,color=black}] (1)--(6)node[pos=0.4,anchor=east]{$1$};
\draw[style={->, thick,color=black}] (2)--(7)node[pos=0.4,anchor=south]{$1$};
\draw[style={->, thick,color=black}] (3)--(8)node[pos=0.4,anchor=east]{$1$};
\draw[style={->, thick,color=black}] (4)--(9)node[pos=0.6,anchor=east]{$1$};
\draw[style={->, thick,color=black}] (5)--(10)node[pos=0.4,anchor=south]{$-2$};

\draw[style={-, thick}] (1)--(3) [dashed] node[pos=0.1,anchor=east]{$-1$};
\draw[style={-, thick}] (2)--(4) [dashed] node[pos=0.1,anchor=north]{$1$};
\draw[style={-, thick}] (3)--(5) [dashed] node[pos=0.17,anchor=north]{$0$};
\draw[style={-, thick}] (4)--(1) [dashed] node[pos=0.1,anchor=west]{$0$};
\draw[style={-, thick}] (5)--(2) [dashed] node[pos=0.1,anchor=south]{$2$};

\draw[style={-, thick}] (6)--(7) [dashed] node[pos=0.5,anchor=east]{\hspace{-7mm}$0$};
\draw[style={-, thick}] (7)--(8) [dashed] node[pos=0.5,anchor=east]{$1$};
\draw[style={-, thick}] (8)--(9) [dashed] node[pos=0.5,anchor=north]{$0$};
\draw[style={-, thick}] (9)--(10) [dashed] node[pos=0.5,anchor=west]{$-1$};
\draw[style={-, thick}] (10)--(6) [dashed] node[pos=0.5,anchor=west]{\hspace{2mm}$1$};
\end{tikzpicture}
\caption{The Petersen graph}\label{fig:Petersen}
\end{figure}

\noindent Let the cyclic order of the vertices $v_1$ and $v_2$ be given by $\rcirclearrowleft$. The other cyclic orders are $\lcirclearrowright$. Then there is one face and the Petersen graph is embedded in the $3$-torus. In Figure \ref{fig:Petersen} a nowhere-identity dihedral $3$-flow $f$ is specified as follows: $f$ evaluates to reflections on dashed edges. The other edges are rotations, and therefore are oriented. The value on an edge $e$ denotes the upper right entry of the matrix $f(e)$.
\end{proof}

\begin{lemma}\label{lemma:ydelta}
Let $\Gamma$ be an embedded cubic graph that has a nowhere-identity dihedral $4$-flow. The cubic graph $\Gamma'$ obtained from $\Gamma$ by replacing any of its vertices by a triangle, is embedded in the same surface as $\Gamma$ and also has a nowhere-identity dihedral $4$-flow.
\end{lemma}
\begin{proof}
By Euler's formula (\ref{equation:euler}), $\Gamma'$ embeds in the same surface as $\Gamma$. A nowhere-identity dihedral $4$-flow $f$ on $\Gamma$ either evaluates to three rotations or to one rotation and two reflections at a vertex. The following figure shows that in both cases we can extend $f$ to a nowhere-identity dihedral $4$-flow on $\Gamma'$.
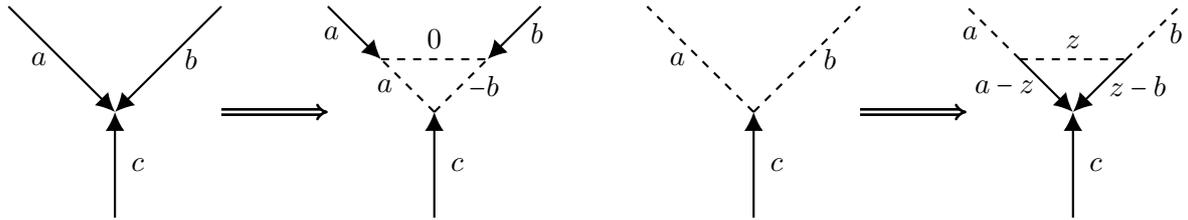
\begin{figure}[H]
\centering
\begin{tikzpicture}
[scale=1.4]
\draw[<-,thick]
(0,0) -- (1,1) node[pos=0.5]{\hspace{6mm}$b$};
\draw[<-,thick]
(0,0) -- (-1,1) node[pos=0.5]{\hspace{-6mm}$a$};
\draw[<-,thick]
(0,0) -- (0,-1) node[pos=0.5]{\hspace{6mm}$c$};

\draw[-Implies, line width=1pt, double distance=1pt]
(1,0) -- (2,0) node[pos=0.5]{};

\draw[style={-, thick}]
(3,0) -- (3.5,0.5) [dashed] node[pos=0.5]{\hspace{6mm}$-b$};
\draw[style={-, thick}]
(3,0) -- (2.5,0.5) [dashed] node[pos=0.5]{\hspace{-6mm}$a$};
\draw[style={-, thick}]
(2.5,0.5) -- (3.5,0.5) [dashed] node[pos=0.5,anchor=south]{\vspace{-6mm}$0$};
\draw[<-,thick]
(3,0) -- (3,-1) node[pos=0.5]{\hspace{6mm}$c$};
\draw[<-,thick]
(3.5,0.5) -- (4,1) node[pos=0.5]{\hspace{6mm}$b$};
\draw[<-,thick]
(2.5,0.5) -- (2,1) node[pos=0.5]{\hspace{-6mm}$a$};

\draw[style={-, thick}]
(6,0) -- (7,1) [dashed] node[pos=0.5]{\hspace{6mm}$b$};
\draw[style={-, thick}]
(6,0) -- (5,1) [dashed] node[pos=0.5]{\hspace{-6mm}$a$};
\draw[<-,thick]
(6,0) -- (6,-1) node[pos=0.5]{\hspace{6mm}$c$};

\draw[-Implies, line width=1pt, double distance=1pt]
(7,0) -- (8,0) node[pos=0.5]{};

\draw[<-,thick]
(9,0) -- (9.5,0.5) node[pos=0.5]{\hspace{10mm}$z-b$};
\draw[<-,thick]
(9,0) -- (8.5,0.5) node[pos=0.5]{\hspace{-11mm}$a-z$};
\draw[style={-, thick}]
(8.5,0.5) -- (9.5,0.5) [dashed] node[pos=0.5,anchor=south]{\vspace{-6mm}$z$};
\draw[<-,thick]
(9,0) -- (9,-1) node[pos=0.5]{\hspace{6mm}$c$};
\draw[style={-, thick}]
(9.5,0.5) -- (10,1) [dashed] node[pos=0.5]{\hspace{6mm}$b$};
\draw[style={-, thick}]
(8.5,0.5) -- (8,1) [dashed] node[pos=0.5]{\hspace{-6mm}$a$};
\end{tikzpicture}
\caption{Replacing a vertex by a triangle}\label{fig:ydelta}
\end{figure}
\noindent In Figure \ref{fig:ydelta} the dashed edges are reflections and the other edges are rotations. The value on an edge $e$ denotes the upper right entry of the matrix $f(e)$. The integer $z$ in the rightmost picture is chosen such that $0 \leq |z| \leq 3, 1 \leq |z-a| \leq 3$ and $1 \leq |z-b| \leq 3$.
\end{proof}

\begin{proposition}\label{proposition:cubics16}
All bridgeless cubic graphs on at most $16$ vertices have a nowhere-identity dihedral $4$-flow. 
\end{proposition}
\begin{proof}
We need only to consider snarks, and bridgeless cubic graphs that are not $3$-edge colorable with girth less than $5$. There are two such graphs with at most $16$ vertices: the Petersen graph and the Tietze graph. The Petersen graph has been dealt with in Proposition \ref{proposition:petersen}. The Tietze graph is obtained from the Petersen graph by replacing one of its vertices by a triangle. Hence, Proposition \ref{proposition:petersen} and Lemma \ref{lemma:ydelta} account for this case.
\end{proof}

\noindent Snarks start to appear again among bridgeless cubic graphs with $18$ or more vertices. The next propositions deal with snarks that have additional properties.\\
\indent Clearly, no snark is Hamiltonian. However, many snarks are \textit{hypohamiltonian}: the removal of any vertex results in a Hamiltonian graph. We call a snark \textit{almost Hamiltonian} if there exists a vertex whose removal yields a Hamiltonian graph. Given a snark $\Gamma = (V,E)$, we define
\begin{equation}\label{equation:ah}
V_{\text{aH}} := \{v \in V \mid \Gamma \setminus v \text{ is Hamiltonian}\},
\end{equation}
so that $\Gamma$ is almost Hamiltonian if and only if $V_{\text{aH}} \neq \emptyset$. Given an embedding of $\Gamma$, call a vertex $v$ \textit{simple} if there exists an edge incident with $v$ that is contained in a simple contractible cycle of the embedded graph $\Gamma$. We define
\begin{equation}\label{equation:vertexsimple}
V_{\text{simple}} := \{v \in V \mid v \text{ is simple}\}.
\end{equation}

\begin{proposition}
Let $\Gamma = (V,E)$ be an embedded snark. If $V_{\text{aH}} \cap V_{\text{simple}} \neq \emptyset$, then $\Gamma$ admits a nowhere-identity dihedral $4$-flow.
\end{proposition}
\begin{proof}
Let $v \in V_{\text{aH}} \cap V_{\text{simple}}$. Let $v_1,v_2$ and $v_3$ be the neighbors of $v$. They are distinct and not adjacent to one another. For $i = 1,2,3$, let the edge $e_i$ be defined by $e_i = vv_i$. Assume that $e_1$ is contained in a simple contractible cycle. Consider a Hamiltonian cycle in $\Gamma \setminus v$. Color the edges incident with $v_1$ in $\Gamma \setminus v$ with color $1$. The other edges on the Hamiltonian cycle (which is of odd length) are colored alternatingly with the colors $1$ and $2$. The remaining edges get the color $3$.\\
\indent Insert the edges $e_2$ and $e_3$ in $\Gamma \setminus v$ and color them with $3$. The obtained graph is a subdivision of a cubic graph $\Gamma'$. The described coloring induces a $3$-edge coloring on $\Gamma'$. Therefore, $\Gamma'$ admits a nowhere-zero $4$-flow and by Theorem \ref{theorem:removal}, $\Gamma$ has a nowhere-identity dihedral $4$-flow.  
\end{proof}

The following property is best expressed in terms of edges. A snark $\Gamma = (V,E)$ for which there exists an edge $e = uv$ such that $\Gamma \setminus \{u,v\}$ is $3$-edge colorable, is called \textit{almost }2\textit{-vertex critical}. We define
\begin{equation}\label{equation:2vc}
E_{\text{avc}} := \{uv = e \in E \mid \Gamma \setminus \{u,v\} \text{ is }3\text{-edge colorable}\}. 
\end{equation}
Given an embedding of $\Gamma$, an edge that is contained in a simple contractible cycle is called \textit{simple}. We define
\begin{equation}\label{equation:edgesimple}
E_{\text{simple}} := \{e \in E \mid e \text{ is simple}\}.
\end{equation}

\begin{proposition}
Let $\Gamma = (V,E)$ be an embedded snark. If $E_{\text{avc}} \cap E_{\text{simple}} \neq \emptyset$, then $\Gamma$ admits a nowhere-identity dihedral $4$-flow.
\end{proposition}
\begin{proof}
Let $uv = e \in E_{\text{avc}} \cap E_{\text{simple}}$, with $u,v \in V$. Let $u_1$ and $u_2$ be the other neighbors of $u$ in $\Gamma$, and $v_1$ and $v_2$ the other neighbors of $v$. They are distinct and there are no edges between them. Consider a $3$-edge coloring of $\Gamma \setminus \{u,v\}$.\\
\indent For $i = 1,2,3$, the color $i$ appears an even number of times among the edges incident with the vertices in $U := \{u_1,u_2,v_1,v_2\}$. Hence, an even number of vertices in $U$ misses color $i$. If $u_1$ and $u_2$ miss the same color, then so do $v_1$ and $v_2$. In that case, the $3$-edge coloring of $\Gamma \setminus \{u,v\}$ can be extended to a $3$-edge coloring of $\Gamma$, which is a contradiction. Therefore, $u_1$ and $u_2$ miss the same color, say color $1$. Let $2$ be the color that $v_1$ and $v_2$ miss.\\
\indent Consider the graph $\Gamma'$ obtained from $\Gamma \setminus \{u,v\}$ by inserting the edges $u_1u_2$ and $v_1v_2$, and color the edges with colors $1$ and $2$ respectively. Then we have a $3$-edge coloring of $\Gamma'$. Hence there exists a nowhere-identity dihedral $4$-flow on $\Gamma$ by Theorem \ref{theorem:removal}.
\end{proof}

In all of the previous results, the presence of a simple contractible cycle or a simple facial walk in an embedded snark is crucial. An embedding of a graph in a surface in which every facial walk is simple, is called a \textit{strong embedding}. For some time, it was thought that any bridgeless cubic graph admits a strong embedding in a surface of its own genus (see for instance \cite{seymour79}). This statement, however, is proven to be false \cite{richter83}. The following conjecture, which is due to Jaeger, is still open.

\begin{conjecture}(Jaeger \cite{jaeger85})\label{conjecture:jaeger}
Every $2$-connected graph $\Gamma$ has a strong embedding in some orientable surface.
\end{conjecture}

\noindent This conjecture has been verified for $2$-connected projective-planar cubic graphs \cite{ellingham11}.

\section{Further research}\label{section:fr}

In this paper we have studied to some extent when statement (\ref{equation:main}) holds and in which cases it does not. We have only dealt with embeddings of graphs in orientable surfaces. It would be interesting to investigate the non-orientable case. In order to do so, it seems that a non-orientable version of Theorem \ref{theorem:devos} is necessary. To the best of our knowledge, such a theorem has not been discovered.\\
\indent In Section \ref{subsection:nis2} we considered the complexity of deciding the existence of nowhere-identity dihedral $2$-flows. It would be interesting to consider the corresponding problem for the cases $n=3$ and $n=4$. More specifically, we would like to know the answer to the following question.

\begin{question}
Does every planar cubic graph have a nowhere-identity dihedral $3$-flow?
\end{question}

\noindent In fact, we propose the following conjecture, whose confirmation would imply a positive answer to the previous question.

\begin{conjecture}\label{conjecture:litjens}
For every $3$-edge colorable cubic graph, there exists an embedding in a surface of its own genus, with respect to which it has a nowhere-identity dihedral $2$-flow.
\end{conjecture}

In Section \ref{subsection:counting} we briefly discussed the enumeration of nowhere-identity dihedral $n$-flows. The question whether the number of such flows is counted by a quasi-polynomial in $n$ of period $2$ remains open, although lattice point counting methods may prove successful, as in~\cite{kochol02}.\\

\noindent \textit{Acknowledgements.} The author is grateful to Lex Schrijver for his contribution, viz.\ Theorem \ref{theorem:removal}. We also thank Sven Polak, Lex and the referees for their valuable comments regarding both the content and the presentation of this paper. Furthermore, we wish to acknowledge the many fruitful discussions with Guus Regts, Bart Sevenster, Jacob Turner and Llu\'is Vena Cros.

\end{document}